\documentclass[12pt]{article}
\usepackage{amsmath,amsthm}
\usepackage{amsfonts,mathrsfs}

\usepackage{color}
\usepackage[colorlinks,anchorcolor=blue,citecolor=blue]{hyperref}

\allowdisplaybreaks

\theoremstyle{plain}
\newtheorem{thm}{\noindent\bf Theorem}
\newtheorem{cor}{\noindent\bf Corollary}
\newtheorem{lem}{\noindent\bf Lemma}
\newtheorem{prop}{\noindent\bf Proposition}
\theoremstyle{remark}
\newtheorem{rmk}{\noindent \bf \textit{Remark}}
\newtheorem{exam}{\noindent \bf Example}

\def\bo{\mathbf{1}}
\newcommand{\D}{\displaystyle}
\newcommand{\non}{\nonumber}
\newcommand{\Pm}{\mathbb{P}}
\newcommand{\Em}{\mathbb{E}}

\newcommand{\wa}{\mathcal{W}_{a}^{(p,q)}}
\newcommand{\wab}{\mathcal{W}_{(a,b)}^{(p,q)}}
\newcommand{\tdc}{\tau_{d}^{+}\wedge\tau_{c}^{-}}
\newcommand{\uab}{u_{(a,b)}^{(p+\lambda,q+\lambda)}}
\newcommand{\sxp}{\sigma_{e_{\lambda}}^{+}}
\newcommand{\sxm}{\sigma_{e_{\lambda}}^{-}}
\newcommand{\sxe}{\sigma_{e_{\lambda}}^{\{0\}}}

\begin{document}

\title{Occupation times of intervals until last passage times for spectrally negative L{\'e}vy processes}
\author{
Bo Li and Chunhao Cai \\
\small School of Mathematics and LPMC, Nankai University}
\date{}
\maketitle

\begin{abstract} 
In this paper, we derive the  Laplace transforms of occupation times of intervals until last passage times for spectrally negative L{\'e}vy processes.
Motivated by \cite{Baurdoux2009:lastexp:levy}, the last times before an independent exponential variable are investigated here. 
By applying the dual argument, explicit formulas are obtained in terms of analytical identities first introduced by Loeffen et al. \cite{Loeffen2014:occupationtime:levy}.
\end{abstract}

\textbf{Keywords:} 
Occupation times; 
spectrally negative L\'evy process; 
last passage times;
scale functions;

%
%

\section{Introduction}
In the risk theory, ruin occurs at the first time when the surplus process becomes negative.
For a surplus modelled by a spectrally negative Levy process (SNLP) which we shall denote by $X=(X_t)_{t\geq0}$, the features at the time of ruin have been throughly studied, see for example Kyprianou \cite{Kyprianou2014:book:levy}.
There are also situations, as studied by \cite{Gerber1990, NokChiu2005:passagetime:levy, Baurdoux2009:lastexp:levy, Albrecher2011:optimal:omega, Gerber2012:omega}, where the moment of ruin may not be the most important character of a risk process. 
For example, in the omega model considered in \cite{Albrecher2011:optimal:omega} and \cite{Gerber2012:omega}, there is a distinction between ruin (negative surplus) and bankruptcy (going out of business).
A portfolio is assumed to be one of many that belongs to a company.
When the reserve for the portfolio becomes negative but not too severe, other funds can be brought to support the negative surplus for a while with the hope that the portfolio will recover in the future. 
In such situation, perhaps a more interesting question is how and how long does the surplus stay negative before bankruptcy takes place. 
Therefore, a more realistic object to study might be the last passage time below $0$ before a fixed time $t$, i.e.
\begin{equation}
\sigma_{t}^{-} =\sup\{0\leq s< t, X_{s}<0\}.
\end{equation}
with the convention that $\sup\emptyset=0$. 
Motivated by \cite{Baurdoux2009:lastexp:levy}, we focus on the last times before $e_{\lambda}$, where $e_{\lambda}$ is an exponential variable with mean $\lambda^{-1}$ and independent of $X$. 
In the main results of this paper, Theorem \ref{main:s-}, we identify the following distribution 
\[
\Em_{x}\left[e^{-L(\sxm)};  X(\sxm)\in dy, 0<\sxm \leq \tdc\right],
\]
where $L(\cdot)$ is the occupation times up to time $t$ which was first considered by \cite{Loeffen2014:occupationtime:levy}:
\begin{equation}\label{def:localtime}
L(t):=\int_0^{t} \left(p \bo_{\{X_{s}\notin(a,b)\}}+ q \bo_{\{X_{s}\in(a,b)}\})\right) ds, 
\end{equation}
where $p,q\geq0$, $d>0>c$, $d\geq b\geq a\geq c$, $x\in(c,d)$ and where
\begin{equation}
\tau_{x}^{+}= \inf\{t\geq 0, X_t>x\}\quad\text{and} \quad \tau_{x}^{-}= \inf\{t\geq 0, X_t<x\}.
\end{equation}
Using similar techniques, we also find the Laplace transform of $L(\cdot)$ at  
\begin{equation}
\sxp=\sup\{0\leq s< e_{\lambda}, X_{s}>0\}\quad \text{and}\quad  
\sxe=\sup\{0\leq s< e_{\lambda}, X_{s}=0\},
\end{equation}
with the convention that $\sup\emptyset=0$. Using the convention that $e_{0}=\infty$ with probability $1$, $(\sigma_{e_{0}}^{+}, \sigma_{e_{0}}^{-}, \sigma_{e_{0}}^{\{0\}})$ 
reduces to the globe last passage times.

In the last few years, several papers have looked at the distribution of functionals involving occupation times of a stochastic process. 
\cite{Landriault2011:occupationtime:levy} computed the Laplace transform of occupation times of the negative half-line of an SNLP,
whereas \cite{Kyprianou2013:occupationtime:Rlevy} studied the same functionals of a refracted L\'evy process. 
On the other hand, \cite{Gerber1990} found the Laplace transform of the last passage time at a certain level for the classical risk process. 
Then \cite{NokChiu2005:passagetime:levy} and \cite{Baurdoux2009:lastexp:levy} extended the results of the last times for an SNLP.
In addition, \cite{Li2013:occupationtime:diffusion} focused on time-homogeneous diffusion process.
For the occupation time of intervals of an SNLP, as defined in \eqref{def:localtime} and initially appeared in \cite{Loeffen2014:occupationtime:levy},
\cite{Loeffen2014:occupationtime:levy} obtained the Laplace transform of the functionals at first passage times. 
Furthermore, \cite{Guerin2014:occupationmeasure:levy} investigated its associating resolvent measures. 
More recently, via adopting the Poisson approach and approximation, \cite{Zhou:last} studied similar problems about the last passage times where the occupation times
on the negative half-line are considered.

In this paper, a dual argument is applied, which is different from the approach adopted in  \cite{Zhou:last}. 
Observing that the last time is dual to the first time by time reversal at $e_{\lambda}$, 
we could avoid dealing with the non-Markov property  of the last times and also helps in distinguishing the events between creeping and jumping in an intuitive way.
Making use of results already known and the strong Markov property, we derive formulas  in terms of  the modified analytical identities as shown in Loeffen et al. \cite{Loeffen2014:occupationtime:levy}.

The rest of this paper is organized as follows. In section 2 we present some primary results related to a L\'evy process as well as several nice properties concerning the functional $L$. The main results are presented in Section 3. Then in Sections 4 our proofs are given.

\section{Premilinaries}

We now briefly review some important properties of an SNLP and the associated scale functions. The reader is referred to the books Bertoin \cite{Bertoin96:book} and Kyprianou \cite{Kyprianou2014:book:levy} for an introduction to the theory of L\'evy processes. Here we exclude the case that $X$ is the negative of a subordinator.

Let $X=\{X_t, t\geq0\}$ be a spectrally negative L\'evy process on $(\Omega, \mathcal{F}, \{\mathcal{F}_{t}\}_{t\geq0}, \Pm)$, namely a process with stationary independent increments and without positive jumps. 
The law of $X$ such that $X_{0}=x$ is denoted by $\Pm_{x}$ and the corresponding expectation by $\Em_{x}$.
For $\lambda\geq0$, define
\[
\Em[\exp(\lambda X_t)]=\exp(t\psi(\lambda)).
\]
Being absent of positive jumps,  
$\psi(\lambda)$ is well defined and known as the Laplace exponent of $X$. It is a continuous and convex function given by the L\'evy-Khintchine formula
\[
\psi(\lambda)=\frac{1}{2}\sigma^2\lambda^2+ \gamma \lambda+ \int_{(0,\infty)} \left(e^{-\lambda x}-1+ \lambda x \bo_{\{x\in(0,1]\}}\right)\Pi(dx),
\]
where $\Pi(dx)$  is called L\'evy measure satisfying $\int_{(0, \infty)}(1\wedge x^2)\Pi(dx)<\infty$. The right continuous inverse of $\psi(\cdot)$ is denoted by 
\[
\phi(s)=\sup\{t\geq 0: \psi(t)=s\}.
\]

In the analysis of an SNLP,  the fluctuation theory  is closely connected to the two-sided exit problem as well as the resolvent measures.
The so-called scale functions $(W^{(q)}, Z^{(q)})$ play a vital role in the exiting formulas.
For $q\geq 0$, $W^{(q)}$ is defined as a continuous and increasing function satisfying
\begin{equation}
\int_0^{\infty} e^{-\theta y} W^{(q)}(y) dy=\frac{1}{\psi(\theta)-q}\quad\text{for $\theta>\phi(q)$},
\end{equation}
and $W^{(q)}(x)=0$ for all $x<0$. Define also
\begin{equation}
Z^{(q)}(x)=1+ q \int_0^{x}W^{(q)}(y) dy, \quad x\in \mathbb{R}.
\end{equation}
We write $W(x)=W^{(0)}(x)$ and $Z(x)=Z^{(0)}(x)=1$ when $q=0$. 


For $c\geq0$, the process $\{\exp(c X_t- \psi(c) t), t\geq0\}$ is a martingale under $\Pm$. We can introduce the change of measure which is another useful tool in solving existing problems,
\[
\left.\frac{d \Pm^{(c)}}{d \Pm} \right|_{\mathcal{F}_t}= e^{c X_t-\psi(c) t} \quad\text{for every $t>0$}.
\]
$X$ is still  an SNLP under $\Pm^{(c)}$. Its Laplace exponent and associatied scale functions under $\Pm^{(c)}$ are marked with a subscript $c$. A straightforward calculation shows that:
\[
\psi_c(s)= \psi(c+ s)- \psi(c)\quad\text{and}\quad
\phi_{c}(s)=\phi(\psi(c)+s)- c, \quad \forall s\geq0,
\]
and for $c\geq0$, $q+\psi(c)\geq0$, 
\[
W_c^{(q)}(x)= e^{-c x} W^{(q+\psi(c))}(x),\quad Z_{c}^{(q)}(x)= 1+ q \int_0^{x} W_{c}^{(q)}(y) dy.
\]
For their limiting behaviours, we have, for $q\geq0$,
\begin{equation}
e^{-\phi(q) x} W^{(q)}(x)\to\phi'(q), \quad \frac{Z^{(q)}(x)}{W^{(q)}(x)}\to \frac{q}{\phi(q)}\quad\text{as $x\to\infty$}. \label{scale:app}
\end{equation}

Applying the fluctuation theory and a change of measure, the following results concerning an SNLP can be found from \cite{Kyprianou2014:book:levy} 

\begin{prop}\label{prop:levy} For $x\geq0$, $q\geq0$, regarding one-sided passage times we have
\[
\Em\left[e^{-q \tau_{x}^{+}}; \tau_{x}^{+}<\infty\right]= e^{-\phi(q) x}, \quad \Em_{x}\left[e^{-q \tau_{0}^{-}}; \tau_{0}^{-}<\infty\right]= Z^{(q)}(x)- \frac{q}{\phi(q)} W^{(q)}(x).
\]
For the two-sided passage times, we have for $0\leq x\leq c$, $q,u,v\geq0$ with $p=u-\psi(v)$
\begin{align*}
\Em_{x}\left[e^{-q \tau_{c}^{+}}; \tau_{c}^{+}\leq \tau_{0}^{-}\right]=&\ \frac{W^{(q)}(x)}{W^{(q)}(c)},\\
\Em_{x}\left[e^{-u \tau_{0}^{-}+ vX(\tau_{0}^{-})}; \tau_{0}^{-}\leq \tau_c^{+}\right]=&\ e^{v x}\left(Z_{v}^{(p)}(x)- W_{v}^{(p)}(x)\frac{Z_{v}^{(p)}(c)}{W_{v}^{(p)}(c)}\right)\\
\Em_{x}\left(e^{- q \tau_{0}^{-}}; X(\tau_{0}^{-})=0\right)=&\ \frac{\sigma^{2}}{2}\left(W^{\prime(q)}(x)-\phi(q) W^{(q)}(x)\right).
\end{align*}
The potential measure of $X$ killed at exiting $[0,\infty)$ is given by
\[
R^{(q)}(x,dy):=\int_0^{\infty} e^{-qt}\Pm_{x}(X_t\in dy, t\leq \tau_{0}^{-})
= \left(e^{-\phi(q) y} W^{(q)}(x)- W^{(q)}(x-y)\right) dy.
\]
Let $\tau^{\{x\}}=\inf\{t>0, X(t)=x\}$ be the first hitting time of level $x$, then for $x\in\mathbb{R}$ 
\[
\Em\left[e^{-q \tau^{\{x\}}}; \tau^{\{x\}}<\infty\right]= e^{-\phi(q) x}- \frac{1}{\phi'(q)}W^{(q)}(-x).
\]
\end{prop}

For the problems concerning the occupation time on intervals, $L(\cdot)$ in \eqref{def:localtime}, \cite{Loeffen2014:occupationtime:levy} found the Laplace transform of $L(\cdot)$ at first passage times and \cite{Guerin2014:occupationmeasure:levy} obtained the associated resolvent measure. In both works, the following auxiliary functions introduced  by \cite{Loeffen2014:occupationtime:levy} are useful 
\begin{align*}
\wa(x):=&\ W^{(p)}(x)+ (q-p) \int_{a}^{x} W^{(q)}(x-z) W^{(p)}(z) dz,\\
\mathcal{H}^{(p,q)}(x):=&\ e^{\phi(p) x} \left(1+ (q-p) \int_{0}^{x} e^{-\phi(p) z} W^{(q)}(z) dz\right).
\end{align*}
To simplify notations, we prefer to use a similar version with level $b$ included
\begin{align*}
\wa(x,y):=&\ W^{(p)}(x-y)+ (q-p) \int_{a}^{x}W^{(q)}(x-z)W^{(p)}(z-y) dz,\\
\wab(x,y):=&\ W^{(p)}(x-y)+ (q-p)\int_{a}^{b}W^{(p)}(x-z)\wa(z,y) dz,\\
\mathcal{H}_{a}^{(p,q)}(x):=&\ e^{\phi(p)x}+ (q-p) \int_{a}^{x} W^{(q)}(x-z) e^{\phi(p)z} dz,\\
\mathcal{H}^{(p,q)}_{(a,b)}(x):=&\ e^{\phi(p)x}+ (q-p)\int_{a}^{b} W^{(p)}(x-z)\mathcal{H}_{a}^{(p,q)}(z) dz.
\end{align*}
It can be checked that, $\D \phi'(q)\mathcal{H}_{(a,b)}^{(p,q)}(x)=\lim_{y\to-\infty} \wab(x,y)e^{\phi(q) y}$ and moreover
\begin{align*}
\wa(x,y)=&\ \mathcal{W}_{a-y}^{(p,q)}(x-y),\quad \mathcal{H}_{a}^{(p,q)}(x)=e^{\phi(p)a}\mathcal{H}^{(p,q)}(x-a),\\
\wab(x,y)=&\ \wa(x,y)- (q-p)\int_{b}^{x}W^{(p)}(x-z)\wa(z,y) dz,\\
e^{-\phi(p)a}\mathcal{H}^{(p,q)}_{(a,b)}(x)=&\ \mathcal{H}^{(p,q)}(x-a)- (q-p)\int_{b}^{x}W^{(p)}(x-z)\mathcal{H}^{(p,q)}(z-a) dz.
\end{align*}


The following results can be found from \cite{Loeffen2014:occupationtime:levy}.
\begin{prop}[ Loeffen et al. \cite{Loeffen2014:occupationtime:levy}]\label{prop:ronnie} 
Let $L$ be defined by \eqref{def:localtime}, for  $x\in[c,d]$, 
\[
\Em_{x}\left[e^{-L(\tau_{d}^{+})}; \tau_{d}^{+}<\tau_{c}^{-}\right]=\frac{\wab(x,c)}{\wab(d,c)}
\quad\text{and}\quad
\Em_{x}\left[e^{-L(\tau_{d}^{+})}; \tau_{d}^{+}<\infty\right]= \frac{\mathcal{H}^{(p,q)}_{(a,b)}(x)}{\mathcal{H}^{(p,q)}_{(a,b)}(c)}.
\]
\end{prop}

The resolvent measure with respect to $L(\cdot)$ is given below, as one can identify the formula with that in \cite{Guerin2014:occupationmeasure:levy} after some calculations,

\begin{prop}[Gu\'erin and Renaud  \cite{Guerin2014:occupationmeasure:levy}]\label{prop:Guerin} For $x, y\in\mathbb{R}$
\begin{align}
&\ \int_{0}^{\infty}\Em_{x}\left[e^{-L(t)}; X(t)\in dy\right] dt\non\\
=&\ \left(\frac{e^{-\phi(p)(a+b)} \mathcal{H}^{(p,q)}_{(a,b)}(x) \mathcal{H}^{(p,q)}_{(a,b)}(a+b-y)}{\psi'(\phi(p))+ (q-p) \int_{0}^{b-a} e^{-\phi(p)z} \mathcal{H}^{(p,q)}(z) dz}- \wab(x,y)\right) dy. 
\end{align}
\end{prop}


\section{Main results}
Recall that $d>0>c$, $d\geq b\geq a\geq c$ and $p,q,\lambda\geq0$ are fixed constants. Since the two-sided passage problem is
%
what we are concerned in the paper, 
we would need the following resolvent measure which extends the results in \cite{Guerin2014:occupationmeasure:levy},

\begin{thm}\label{main:resol} Let $f$ be a nonnegative measurable function,  we have for $x\in[c,d]$
\begin{align}
U^{(p,q)}_{(a,b)}f(x):=&\ \int_{0}^{\infty} \Em_{x}\left[e^{-L(t)}  f(X_{t}); t\leq \tdc\right] dt\non\\
=&\ \int_{c}^{d} f(y) \left(\frac{\wab(x,c)}{\wab(d,c)} \wab(d,y) - \wab(x,y)\right) dy.
\end{align}
\end{thm}

We can see that the measure $U^{(p,q)}_{(a,b)}(x,dy)$ is absolutely continuous on its support $[c,d]$, and its density function is given by 
\begin{equation}
u^{(p,q)}_{(a,b)}(x,y)=\left(\frac{\wab(x,c)}{\wab(d,c)} \wab(d,y) - \wab(x,y)\right). \label{eqn:resol:denstity}
\end{equation}
When $p=q$, $\wab(x,y)= W^{(p)}(x-y)$, and \eqref{eqn:resol:denstity} is then reduced to the classical case as shown in Proposition \ref{prop:levy} on $[c,d]$.
We also need $\mu^{(\lambda)}(dy)$ defined below
which is the $\lambda$-capacity measure of $\mathbb{R}^{-}$,

\begin{lem}\label{lem:mu} For $\lambda>0$, let $\mu^{(\lambda)}(dy)$ be the measure on $\mathbb{R}$ defined by
\[
\mu^{(\lambda)}(dy)= \lambda \int_{\mathbb{R}} \Em_{z}\left[e^{-\lambda \tau_{0}^{-}}; (-X(\tau_{0}^{-}))\in dy, \tau_{0}^{-}<\infty\right] dz,
\]
then $\mu^{(\lambda)}(\cdot)$ is a Randon measure concentrating on  $[0,\infty)$ and for $s\geq0$
\begin{equation}
\widehat{\mu^{(\lambda)}}(s)=\int_{\mathbb{R}^{+}} e^{-s y} \mu^{(\lambda)}(dy)=\frac{\phi(\lambda)(\psi(s)- \lambda)}{s(s-\phi(\lambda))}.
\end{equation}
Taking $s\to\infty$, we have $\mu^{(\lambda)}(\{0\})=\frac{\sigma^{2}}{2}\phi(\lambda)$.
\end{lem}

We are now ready to state our two main results.  

\begin{thm}\label{main:s+} For $x\in[c,d]$, we have for $y\in[0, d)$
\begin{equation}
\Em_{x}\left[e^{-L(\sxp)}; X(\sxp-)\in dy, 0<\sxp\leq \tdc\right] 
 = \uab(x,y) \mu^{(\lambda)}(dy),
\end{equation}
where $\uab(x,y)$ is the resolvent density as in \eqref{eqn:resol:denstity}  with $(p,q)$ replaced by $(p+\lambda, q+\lambda)$. 
In particular, we have from Lemma \ref{lem:mu} that
\begin{equation}
\Em_{x}\left[e^{-L(\sxp)}; X(\sxp-)=0, 0<\sxp\leq \tdc\right]
= \frac{\sigma^{2}}{2} \phi(\lambda)  \uab(x,0).
\end{equation}
\end{thm}

\begin{thm}\label{main:s-}  For $x\in[c,d]$, we have for $y\in(c,0)$
\begin{align}
\Em_{x}\left[e^{-L(\sxm)}; X(\sxm)\in dy; 0<\sxm\leq\tdc\right]=&\  \lambda \uab(x,y) dy,\\
\Em_{x}\left[e^{-L(\sxm)}; X(\sxm)=0, 0<\sxm\leq\tdc\right]=&\  \frac{\lambda}{\phi(\lambda)} \uab(x,0).\label{eqn:main:s-:=}
\end{align}
\end{thm}

Beside the last passage time to a given level, saying $0$ here, another interesting last time would be the last hitting time $\sxe$ which was also studied in \cite{Baurdoux2009:lastexp:levy}.

\begin{thm}\label{main:s=} For $x\in[c,d]$, we have
\begin{equation}
\Em_{x}\left[e^{-L(\sxe)}; 0<\sxe\leq \tdc\right]=\frac{1}{\phi'(\lambda)} \uab(x,0).
\end{equation}
\end{thm}

As a complement, the following results are not hard to find.
\begin{prop}\label{prop:s0} For $x\in\mathbb{R}$,
\begin{align*}
\Pm_{x}(\sxp=0)=&\ \Pm_{x}(\tau_{0}^{+}>e_{\lambda})
= (1-e^{\phi(\lambda)x})\bo_{\{(x<0\}},\\
\Pm_{x}(\sxm=0)=&\ \Pm_{x}(\tau_{0}^{-}>e_{\lambda})
= 1- Z^{(\lambda)}(x)+ \frac{\lambda}{\phi(\lambda)} W^{(\lambda)}(x),\\
\Pm_{x}(\sxe=0)=&\ \Pm_{x}(\tau^{\{0\}}>e_{\lambda})
= 1- e^{\phi(\lambda)x}+ \frac{1}{\phi'(\lambda)}W^{(\lambda)}(x).
\end{align*}
\end{prop}

Being absent of positive jumps,
$\{X(\sxm)<0\}=\{\sxm=e_{\lambda}\}=\{X(e_{\lambda})<0\}$, the first statement in Theorem \ref{main:s-} is not surprising. 
Similarly, applying Theorem \ref{main:resol} and Theorem \ref{main:s+} 
on the set $\{\sxp= e_{\lambda}\}=\{X(e_{\lambda})>0\}$, we could have the following joint distributions when the last positive time is caused by a jump.

\begin{cor}\label{cor:1}For $x\in[c,d]$, we have for $y\in(0,d)$
\begin{align}
&\ \Em_{x}\left[e^{-L(\sxp)}; X(\sxp-)\in dy,  \sxp< e_{\lambda}, 0<\sxp\leq \tdc\right]\non\\
=&\ \uab(x,y) \left( \mu^{(\lambda)}(dy)- \lambda  dy\right). 
\end{align}
\end{cor}

In our main results, the resolvent density $u^{(p,q)}_{(a,b)}(x,y)$ plays predominate roles in all the formulas involved. 
Here, we provide a second way of studying $\sxp$ which could explain the scenarios for $dy$-terms. 
Similar conclusions can be derived for $\sxm$ and $\sxe$. 
Therefore, the events of creeping $0$ for the last times are always the heart of the problem.

\begin{rmk}\label{rmk:2}
Actually, on the set $\{0<\sxp<e_{\lambda}\}$, $X$ creeps $0$ continuously or jumps across level $0$ from somewhere above at the last time $\sxp$. 
For the later case, a negative jump takes place at  some time $t<e_{\lambda}$ such that $X$ fails to regain level $0$ within the rest of time $(t, e_{\lambda})$, and $t$ is then labeled $\sxp$ by definition. An application of the compensation formula yields a second formula
\begin{align*}
&\ \Em_{x}\left[e^{-L(\sxp)}; X(\sxp-)\in dy, 0<\sxp< e_{\theta}, \sxp\leq \tdc\right]\\
=&\ \int_{0}^{\infty} \Em_{x}\left[e^{-L(t)}\cdot 1_{\{e_{\lambda}-t< \tau_{0}^{+}\circ\theta_{t}\}};
X_{t-}\in dy, X_{t}\neq X_{t-}, t<e_{\lambda}\wedge\tdc\right] dt\\
=&\ \uab(x,y)\left(\int_{z>y} \left(1-e^{\phi(\lambda)(y-z)}\right)  \Pi(dz)  \right) dy.
\end{align*}
where $\theta_{\cdot}$ is the shifting operator of $X$.
With Proposition \ref{prop:iden:measure} applied, we have Corollary \ref{cor:1} proved and so is the $dy$ part in Theorem \ref{main:s+}.
\end{rmk}

We claim that the measures in Corollary \ref{cor:1} and Remark \ref{rmk:2} coincide on $(0, \infty)$, that is
\begin{prop}\label{prop:iden:measure} Let $\mu^{(\lambda)}(dy)$ be the measure defined in Lemma  \ref{lem:mu}, for $y>0$, 
\begin{equation}
\left(\mu^{(\lambda)}(dy)- \lambda  dy\right)=\nu(dy):=\left(\int_{z> y} \left(1-e^{\phi(\lambda)(y-z)}\right)  \Pi(dz)  \right) dy. \label{eqn:s+:iden}
\end{equation}
\end{prop}
\begin{proof}[Proof of Proposition \ref{prop:iden:measure}] Recall that the L\'evy-Khintchine formula is given by 
\[
\psi(s)=\frac{1}{2}\sigma^2s^2+ \gamma s+ \int_{(0,\infty)}\left(e^{-s x}-1+ s x \bo_{\{x\leq 1 \}}\right)\Pi(dx), \quad \text{for $s\geq0$}.
\]
Taking Laplace transform of $\nu(\cdot)$ in \eqref{eqn:s+:iden},  we have
\begin{align*}
&\ \widehat{\nu}(s)= \int_{0}^{\infty} e^{-sy} \left(\int_{z> y} \left(1-e^{\phi(\lambda)(y-z)}\right)  \Pi(dz)  \right) dy\\
=&\ \int_{z>0} \Pi(dz) \int_{0}^{z} \left(e^{-s y}- e^{(\phi(\lambda)-s)y-\phi(\lambda)z}\right) dy\\
=&\ \int_{z>0} \Pi(dz) \left((1-e^{-sz})(\frac{1}{s}-\frac{1}{s-\phi(\lambda)})+ (1-e^{-\phi(\lambda)z}) \frac{1}{s-\phi(\lambda)}\right)\\
=&\ \int_{z>0} \Pi(dz) \left(\frac{e^{-sz}-1+ sz \bo_{\{z\leq 1\}}}{s(s-\phi(\lambda))}\phi(\lambda)- \frac{e^{-\phi(\lambda)z}-1+ \phi(\lambda)z \bo_{\{z\leq 1\}}}{s-\phi(\lambda)}\right)\\
=&\ \phi(\lambda)  \frac{\psi(s)- \sigma^{2} s^{2}/2- \gamma s}{s(s-\phi(\lambda))}- \frac{\psi(\phi(\lambda))- \sigma^{2}\phi^{2}(\lambda)/2- \gamma \phi(\lambda)}{s-\phi(\lambda)}\\
=&\  \frac{\phi(\lambda)(\psi(s)-\lambda)}{s(s-\phi(\lambda))}- \frac{\lambda}{s}- \frac{\sigma^{2}}{2}\phi(\lambda),
\end{align*}
which equals to the Laplace transform of $\left(\mu^{(\lambda)}(dy)- \lambda  dy\right)\bo_{\{y>0\}}$  on the left side of \eqref{eqn:s+:iden} and this completes the proof. 
\end{proof}

Integrating with respective to $dy$ in Theorem \ref{main:s+} and \ref{main:s-} over their available domains gives the joint 
Laplace transforms 
of occupation times before last passage times.
\begin{cor}\label{cor:2} For $x\in[c,d]$, 
\begin{align*}
\Em_{x}\left[e^{-L(\sxp)}; 0<\sxp\leq \tdc\right]=&\ 
\int_{0-}^{d} \uab(x,y) \mu^{(\lambda)}(dy),\\
\Em_{x}\left[e^{-L(\sxm)}; 0<\sxm\leq \tdc\right]=&\ 
\frac{\lambda}{\phi(\lambda)} \uab(x,0)+ \lambda \int_{c}^{0} \uab(x,y) dy. 
\end{align*}
\end{cor}

Let $\lambda\to0+$, $\sxp$, $\sxm$  and $\sxe$ increase to 
\[
\sigma_{\infty}^{+}=\sup\{t>0, X_{t}>0\}, \ \
\sigma_{\infty}^{-}=\sup\{t>0, X_{t}<0\} \ \ \text{and}\ \ 
\sigma_{\infty}^{\{0\}}=\sup\{t>0, X_{t}=0\},
\]
respectively. Then we have 
\begin{cor} \label{cor:3}
If $\psi'(0)=0$, then $X$ oscillates, and $\sigma_{\infty}^{-}=\sigma_{\infty}^{+}=\sigma_{\infty}^{\{0\}}=\infty$.\\
If $\psi'(0)>0$, $X\to\infty$, then $\sigma_{\infty}^{+}=\infty$ 
and $\sigma_{\infty}^{\{0\}}=\sigma_{\infty}^{-}$ on $\{\sigma_{\infty}^{-}>0\}$
\[
\Em_{x}\left[e^{-L(\sigma_{\infty}^{-})}; 0<\sigma_{\infty}^{-}\leq \tdc\right]= \psi'(0) u^{(p,q)}_{(a,b)}(x,0).
\]
If $\psi'(0)<0$, $X\to-\infty$, then $\sigma_{\infty}^{-}=\infty$ and $\sigma_{\infty}^{\{0\}}\leq \sigma_{\infty}^{+}<\infty$  on the set $\{\sigma_{\infty}^{+}>0\}$
\begin{align*}
\Em_{x}\left[e^{-L(\sigma_{\infty}^{+})}; 0<\sigma_{\infty}^{+}\leq \tdc\right]=&\ \int_{0-}^{d} u^{(p,q)}_{(a,b)}(x,y)\mu(dy),\\
\Em_{x}\left[e^{-L(\sigma_{\infty}^{\{0\}})}; 0<\sigma_{\infty}^{\{0\}}\leq \tdc\right]=&\ \psi'(\phi(0)) u^{(p,q)}_{(a,b)}(x,0).
\end{align*}
where $\widehat{\mu}(s)=\frac{\phi(0)  \psi(s)}{s(s-\phi(0))}$ for $s>0$.
\end{cor}

With conclusions above, some other occupation times are also available, 
i.e. $(L(\sxp)-L(\tau^{+}_{0}))^{+}$, $(L(\sxe)-L(\tau_{0}^{-}))^{+}$ and  $(L(\sxp)-L(\tau^{\{0\}}))^{+}$,
by applying the strong Markov  property of $X$. 
Similar questions on the differences of times are studied by \cite{NokChiu2005:passagetime:levy} and \cite{Baurdoux2009:lastexp:levy}.
With time reversal approach applied, we could also have the distributions of differences between last times. The following corollaries can be proved following the exact procedures.

\begin{cor}\label{cor:4} For $y>0$ and $z\neq 0$, we have the joint distributions
\begin{align*}
&\ \Em_{x}\left[e^{-L(\sxp)}; X(\sxp)=0, -X(e_{\lambda})\in dy, 0<\sxp\leq\tdc\right]\\
=&\  \lambda\cdot\uab(x,0)\cdot \frac{\sigma^{2}}{2} \left(W^{(\lambda)\prime}(y)-\phi(\lambda)W^{(\lambda)}(y)\right)\cdot dy,\\
&\ \Em_{x}\left[e^{-L(\sxm)}; X(\sxm)=0, X(e_{\lambda})\in dy, 0<\sxm\leq\tdc\right]\\
=&\ \lambda\cdot\uab(x,0)\cdot e^{-\phi(\lambda)y}\cdot dy,\\
&\ \Em_{x}\left[e^{-L(\sxe)}; X(e_{\lambda})\in dz, 0<\sxe\leq\tdc\right]\\
=&\ \lambda\cdot\uab(x,0)\cdot \left(e^{\phi(q)z}-\frac{1}{\phi'(\lambda)} W^{(\lambda)}(z)\right)\cdot dz.
\end{align*}

In addition, the Laplace transform of the difference between last passage times 
are given by, for $x,y>0$ with $x\neq y$ and $d>z>0$
\begin{align*}
&\ \Em\left[e^{L(\sxe)-L(\sxp)}; X_{\sxp-}\in dz, -X_{\sxp}\in dy, -X_{e_{\lambda}}\in dx, 0<\sxe<\sxp\leq \tdc\right]\\
=&\ \lambda\cdot (e^{-\phi(\lambda)y}W^{(\lambda)}(x)-W^{(\lambda)}(x-y)) \cdot \frac{\mathcal{W}^{(p+\lambda,q+\lambda)}_{(a,b)}(d,z)}{\mathcal{W}^{(p+\lambda,q+\lambda)}_{(a,b)}(d,0)}\\
&\ \quad\times \left(\frac{W^{(\lambda)}(-c)}{W^{(\lambda)}(d-c)}W^{(\lambda)}(d)-W(0)\right)\cdot \,dx\,dy\,\Pi(dz+y).
\end{align*}
\end{cor}

We conclude the section by replicating the results in \cite{Baurdoux2009:lastexp:levy}  in which $p=q$. For this case, $\mathcal{W}^{(p+\lambda,q+\lambda)}_{(a,b)}(x,y)=W^{(p+\lambda)}(x-y)$,  $\forall x,y\in\mathbb{R}$.  Specifically, they are demonstrated by the following examples.

\begin{exam}\label{exam:c+d} For $p\geq0$ and $x\in\mathbb{R}$,
\begin{align}
 \Em_{x}\left[e^{-p \sxp}\right]
=&\  \frac{p \phi(\lambda) \phi'(p+\lambda) }{(\phi(p+\lambda)-\phi(\lambda))\phi(p+\lambda)} e^{\phi(p+\lambda)x}  \non\\
&\ \quad-\left(e^{\phi(\lambda)x} Z^{(p)}_{\phi(\lambda)}(x)- \frac{p}{p+\lambda} Z^{(p+\lambda)}(x)-\frac{\lambda}{p+\lambda}\right), 
\label{eqn:exam:1}\\
 \Em_{x}\left[e^{-p \sxm} \right]
=&\ \phi'(p+\lambda) e^{\phi(p+\lambda)x} \left( \frac{\lambda}{\phi(\lambda)}- \frac{\lambda}{\phi(p+\lambda)}\right)
+\frac{\lambda}{p+\lambda} Z^{(p+\lambda)}(x)- \frac{\lambda}{\phi(\lambda)} W^{(p+\lambda)}(x) \non\\
&\ \quad + \left(1- Z^{(\lambda)}(x)+ \frac{\lambda}{\phi(\lambda)} W^{(\lambda)}(x)\right),
\label{eqn:exam:2}\\
 \Em_{x}\left[e^{-p \sxe}\right] 
=&\ \left(1- e^{\phi(\lambda)x}+ \frac{1}{\phi'(\lambda)}W^{(\lambda)}(x)\right)+ \frac{1}{\phi'(\lambda)} \left(e^{\phi(p+\lambda)x} \phi'(p+\lambda)- W^{(p+\lambda)}(x)\right).
\label{eqn:exam:3}
\end{align}
\end{exam}

\begin{proof}[Proof of Example \ref{exam:c+d}] 
It follows directly from Theorem  \ref{main:s=} and Corollary \ref{cor:2} that 
\begin{align*}
&\ \Em_{x}\left[e^{-p \sxe}; 0<\sxe\leq \tau_{d}^{+}\wedge\tau_{c}^{-}\right]
= \frac{1}{\phi'(\lambda)} \left( \frac{W^{(p+\lambda)}(x-c)}{W^{(p+\lambda)}(d-c)}W^{(p+\lambda)}(d)- W^{(p+\lambda)}(x)\right),\\
&\ \Em_{x}\left[e^{- p \sxm}; 0<\sxm\leq\tau_{d}^{+}\wedge\tau_{c}^{-}\right]\non\\
=&\ \frac{\lambda}{\phi(\lambda)}\left(\frac{W^{(p+\lambda)}(x-c)}{W^{(p+\lambda)}(d-c)}W^{(p+\lambda)}(d)- W^{(p+\lambda)}(x)\right)\non\\
&\quad -\frac{\lambda}{p+\lambda}\left(\frac{W^{(p+\lambda)}(x-c)}{W^{(p+\lambda)}(d-c)} \left(Z^{(p+\lambda)}(d)-Z^{(p+\lambda)}(d-c)\right)- \left(Z^{(p+\lambda)}(x)-Z^{(p+\lambda)}(x-c)\right)\right).
\end{align*}
Letting $d\to\infty$, $c\to-\infty$, with the limiting identities \eqref{scale:app} of the scale functions, we replicate the identities \eqref{eqn:exam:2} and \eqref{eqn:exam:3} with additional terms  from Proposition \ref{prop:s0}.

While for equation \eqref{eqn:exam:1}, taking the Laplace transform yields
\begin{align*}
\int_{0}^{\infty} e^{-s u} W^{(p+\lambda)}*\mu^{(\lambda)}(u) du
=&\ \frac{\phi(\lambda)(\psi(s)-\lambda)}{(\psi(s)-p-\lambda)(s-\phi(\lambda))s}\\
=&\ \frac{\psi(s)-\lambda}{(\psi(s)-p-\lambda)(s-\phi(\lambda))}- \frac{p}{p+\lambda} \frac{\psi(s)}{(\psi(s)-p-\lambda) s}- \frac{\lambda}{s(p+\lambda)},\\
W^{(p+\lambda)}*\mu^{(\lambda)}(u) 
=&\ e^{\phi(\lambda) u} Z^{(p)}_{\phi(\lambda)}(u)- \frac{p}{p+\lambda} Z^{(p+\lambda)}(u)- \frac{\lambda}{p+\lambda},
\end{align*}
for $u\geq0$. Thus we have from Theorem \ref{main:s+}, for $x\in[c,d]$
\begin{align*}
\Em_{x}\left[e^{-p \sxp}; 0<\sxp\leq \tau_{d}^{+}\wedge\tau_{c}^{-}\right]
=&\ 
\frac{W^{(p+\lambda)}(x-c)}{W^{(p+\lambda)}(d-c)}
\left(e^{\phi(\lambda) d} Z^{(p)}_{\phi(\lambda)}(d)- \frac{p}{p+\lambda} Z^{(p+\lambda)}(d)- \frac{\lambda}{p+\lambda}\right)\\
&\quad -\left(e^{\phi(\lambda) x} Z^{(p)}_{\phi(\lambda)}(x)- \frac{p}{p+\lambda} Z^{(p+\lambda)}(x)- \frac{\lambda}{p+\lambda}\right)\cdot \bo_{\{x\geq 0\}}
\end{align*}
With an additional term of $\Pm_{x}(\sxp=0)$ from Proposition \ref{prop:s0} and applying the limiting identity \eqref{scale:app}, we have \eqref{eqn:exam:1} holds for $x\in\mathbb{R}$. 
\end{proof}

\section{Proof of main results}
This section will be dedicated to showing proofs for our main results discussed in the previous section.

\begin{proof}[Proof of Theorem \ref{main:resol}] 
To find the resolvent measure of $X$ killed at exiting $[c,d]$, we apply the strong Markov property and Propositions \ref{prop:ronnie} to Propositions \ref{prop:Guerin}.

Firstly, by applying the strong Markov property at $\tau_{d}^{+}$, we have 
\begin{align}
&\ \int_0^{\infty} \Em_{x}\left[e^{-L(t)}; X(t)\in dy, t\leq \tau_{d}^{+}\right] dt\non\\
=&\ \int_0^{\infty} \Em_{x}\left[e^{-L(t)}; X(t)\in dy\right] dt- \int_0^{\infty} \Em_{x}\left[e^{-L(t)}; X(t)\in dy, \tau_{d}^{+}<t\right] dt\non\\
=&\ \int_0^{\infty} \Em_{x}\left[e^{-L(t)}; X(t)\in dy\right] dt- \frac{\mathcal{H}^{(p,q)}_{(a,b)}(x)}{\mathcal{H}^{(p,q)}_{(a,b)}(d)} \times \int_0^{\infty} \Em_{d}\left[e^{-L(t)}; X(t)\in dy\right] dt\non\\
=&\ \left(\frac{\mathcal{H}^{(p,q)}_{(a,b)}(x)}{\mathcal{H}^{(p,q)}_{(a,b)}(d)}\wab(d,y)- W^{(p,q)}_{(a,b)}(x,y)\right) dy, \quad\quad\quad \text{ for $x,y\leq d$. }\label{eqn:inproof:resol:1}
\end{align}

Observe that $\tau^{\{c\}}$, the first hitting time of level $c$ in Proposition \ref{prop:levy}, is a stopping time, and more interestingly, as noticed in \cite{Ivanovs2012:occupationdensity:map}, for an SNLP,
\[
\{\tau_{c}^{-}\leq \tau_{d}^{+}<\infty\}=\{\tau^{\{c\}}\leq \tau_{d}^{+}<\infty\}.
\]
Applying the strong Markov  property at $\tau^{\{c\}}$, we have for $x\leq d$
\begin{align*}
&\ \Em_{x}\left[e^{-L(\tau_{d}^{+})}; \tau_{d}^{+}<\infty\right]
= \Em_{x}\left[e^{-L(\tau_{d}^{+})}; \tau_{d}^{+}\leq \tau_{c}^{-}\right]
+ \Em_{x}\left[e^{-L(\tau_{d}^{+})};  \tau_{c}^{-}\leq \tau_{d}^{+}<\infty\right]\\
=&\ \Em_{x}\left[e^{-L(\tau_{d}^{+})}; \tau_{d}^{+}\leq \tau_{c}^{-}\right]
+ \Em_{x}\left[e^{-L(\tau_{d}^{+})};  \tau^{\{c\}}\leq \tau_{d}^{+}<\infty\right]\\
=&\ \Em_{x}\left[e^{-L(\tau_{d}^{+})}; \tau_{d}^{+}\leq \tau_{c}^{-}\right]+ \Em_{x}\left[e^{-L(\tau^{\{c\}})}; \tau^{\{c\}}\leq \tau_{d}^{+}\right] \Em_{c}\left[e^{-L(\tau_{d}^{+})}; \tau_{d}^{+}<\infty\right].
\end{align*}
Plugging Proposition \ref{prop:ronnie} into the equation above gives
\begin{equation}
\Em_{x}\left[e^{-L(\tau^{\{c\}})}; \tau^{\{c\}}\leq \tau_{d}^{+}\right]= \frac{\mathcal{H}^{(p,q)}_{(a,b)}(x)}{\mathcal{H}^{(p,q)}_{(a,b)}(c)}- \frac{\mathcal{H}^{(p,q)}_{(a,b)}(d)}{\mathcal{H}^{(p,q)}_{(a,b)}(c)}\frac{\wab(x,c)}{\wab(d,c)}, \quad \text{for $x\leq d$}.\label{eqn:inproof:resol:2}
\end{equation}

Finally, substituting \eqref{eqn:inproof:resol:2} into \eqref{eqn:inproof:resol:1} and  applying the strong Markov  property at $\tau^{\{c\}}$ again, for $x, y\in [c,d]$, we have
\begin{align*}
U^{(p,q)}_{(a,b)}&\ (x,dy):= \int_0^{\infty} \Em_{x}\left[e^{-L(t)}; X(t)\in dy, t\leq \tdc\right] dt\\
=&\  \int_0^{\infty} \left(\Em_{x}\left[e^{-L(t)}; X_{t}\in dy, t\leq \tau_{d}^{+}\right]
-\Em_{x}\left[e^{-L(t)}; X_{t}\in dy, \tau_{c}^{-}<t\leq \tau_{d}^{+}\right]\right) dt\\
=&\  \int_0^{\infty} \left(\Em_{x}\left[e^{-L(t)}; X_{t}\in dy, t\leq \tau_{d}^{+}\right]
-\Em_{x}\left[e^{-L(t)}; X_{t}\in dy, \tau^{\{c\}}<t\leq \tau_{d}^{+}\right]\right) dt\\
=&\ \int_0^{\infty} \Em_{x}\left[e^{-L(t)}; X(t)\in dy, t\leq \tau_{d}^{+}\right] dt\\
&\ -\Em_{x}\left[e^{-L(\tau^{\{c\}})}; \tau^{\{c\}}\leq\tau_{d}^{+}\right]\times \int_0^{\infty} \Em_{c}\left[e^{-L(t)}, X(t)\in dy, t\leq \tau_{d}^{+}\right] dt\\
=&\ \left(\frac{\wab(x,c)}{\wab(d,c)} \wab(d,y) - \wab(x,y)\right) dy,
\end{align*}
where the fact that for $y>c$, $\{\tau_{c}^{-}<t\}=\{\tau^{\{c\}}<t\}$ on the set $\{X(t)=y\}$ is used in the third identity and $W^{(p,q)}_{(a,b)}(c,y)=0$ is used in the last line. 
\end{proof}

\begin{proof}[Proof of Lemma \ref{lem:mu}] The Laplace transform of $\lambda$-capacity measure $\mu^{(\lambda)}(dy)$ is derived by applying Proposition \ref{prop:levy} and
the change of measure.

Firstly, for $u,v\geq0$ with $r=u-\psi(v)>0$ and large $t>0$, we have
\begin{align*}
&\ \int_{0}^{\infty} e^{-t x} \Em_{x}\left[e^{-u \tau_{0}^{-}+ v X(\tau_{0}^{-})}; \tau_{0}^{-}<\infty\right] dx\\
=&\ \int_{0}^{\infty} e^{-(t-v) x} \left(Z_{v}^{(r)}(x)- \frac{r}{\phi_{v}(r)}W_{v}^{(r)}(x)\right) dx\\
=&\ \frac{\psi_{v}(t-v)}{(t-v)(\psi_{v}(t-v)-r)}- \frac{r}{\phi_{v}(r)} \frac{1}{\psi_{v}(t-v)-r}\\
=&\ \frac{(\psi(t)-\psi(v))\cdot(\phi(u)-v)- (u-\psi(v))\cdot(t-v)}{(t-v)\cdot (\psi(t)-u)\cdot (\phi(u)-v)}.
\end{align*}
Then the identity holds for $u,v,t>0$ by analytical extension. Particularly, 
\[
\int_{0}^{\infty} \Em_{x}\left[e^{-u \tau_{0}^{-}+ v X(\tau_{0}^{-})}; \tau_{0}^{-}<\infty\right] dx=\frac{v(u-\psi(v))- \psi(v)(\phi(u)-v)}{u\cdot v\cdot(\phi(u)-v)}.
\]
Therefore for $\lambda,s>0$, we have
\begin{align}
&\ \widehat{\mu^{(\lambda)}}(s):= \lambda \int_{\mathbb{R}}  \Em_{x}\left[e^{-\lambda \tau_{0}^{-}+ s X(\tau_{0}^{-})}; \tau_{0}^{-}<\infty\right] dx\non\\
=&\ \frac{\lambda}{s}+ \frac{\lambda-\psi(s)}{\phi(\lambda)-s}- \frac{\psi(s)}{s}
=\frac{\phi(\lambda)\cdot(\psi(s)-\lambda)}{s\cdot(s-\phi(\lambda))},
\end{align}
which gives the formula for $\mu^{(\lambda)}(dy)$ in the Lemma. This completes the proof.
\end{proof}


The proofs of our main results are motivated by the fact that the last times are dual to the first times by time reversal at $e_{\lambda}$ and greatly rely on the dual argument of an SNLP. 
It is well known that, the analytic notion of duality is related to the probabilistic notion of  time reversal for a Markov process. 
Fortunately, things become much simpler for an SNLP. 
In what follows, $\widehat{X}=-X$ denotes the dual process  of $X$ and $\widetilde{X}=(X_{(t-s)-}, 0\leq s\leq t)$ is the time-reversed process for some fixed time $t$. 
In the mathematical notations, a hat $\widehat{\ }$ is used over the existing notations for the characteristics of the dual process, and $\widetilde{\ }$ for those of the reversed process. 
For instance, $\widehat{\Pm}$ stands for the law of $-X$. For every $x\in\mathbb{R}$, $\widehat{\Pm}_{x}$ denotes the law of $x+X$ under $\widehat{\Pm}$, that is the law of $x-X$ under $\Pm$ and also the law of $\widehat{X}$ under $\Pm_{-x}$.  
Before moving onto the main proofs, we need to present the following propositions first which can be found from Chapter III in \cite{Bertoin96:book}.

\begin{prop} \label{prop:inproof:1}
Let $f$ and $g$ be two nonnegative measurable functions, we have for every $t\geq0$, 
\[
\int \Pm_tf(x) g(x) dx= \int f(y) \widehat{\Pm}_t g(y) dy.
\]
\end{prop}
\begin{prop} \label{prop:inproof:2}
For every $t\geq0$, the reversed process $(X_{(t-s)-}-X_t, 0\leq s\leq t)$ and the dual process $(\widehat{X}_s, 0\leq s\leq t)$ has the same law under $\Pm$.
\end{prop}
\begin{prop} \label{prop:inproof:3}
For every $x,y\in\mathbb{R}$, the law of reversed process $(X_{(t-s)-}, 0\leq s\leq t)$ under $\Pm_{x}(\cdot|X(t)=y)$ is a version of the conditional law of $(X_s, 0\leq s\leq t)$ under $\widehat{\Pm}_{y}(\cdot|X_t=x)$.
\end{prop}

Since the dual argument holds for every  $t>0$,  
the Propositions remain valid with $t$ replaced by an independent $e_{\lambda}$.
Thus, we denote by $\widetilde{X}=(X_{(e_{\lambda}-t)-}, 0<t<e_{\lambda})$ the process reversed at $e_{\lambda}$ instead of $t$ in the following proofs.
We also introduce the notation 
\begin{equation}
\widetilde{L}(s):=\int_0^{s}\omega(\widetilde{X}(r)) dr=\int_{0}^{s} \omega(X(e_{\lambda}-r)) dr=L(e_{\lambda})-L(e_{\lambda}-s),
\end{equation}
for $0<s\leq e_{\lambda}$ for simplicity.
We are now ready to prove our main results. 
\begin{proof}[Proof of Theorem \ref{main:s+}] 
As is often the case, we always focus on the integrals with respect to arbitrary nonnegative measurable functions $f, g$  on $\mathbb{R}$.

Firstly, we claim that for $z\in\mathbb{R}$
\begin{equation}
 \int_{0}^{\infty} e^{-\lambda t} \widehat{\Em}_{z} \left(e^{-L(t)} f(X(t)); t\leq \tdc \right) \,dt= \int_{\mathbb{R}} f(x) u^{(p+\lambda, q+\lambda)}_{(a,b)}(x,z)\,dx. \label{eqn:inproof:dual}
\end{equation}
where $u^{(p,q)}_{(a,b)}(x,y)$ is the resolvent density defined in \eqref{eqn:resol:denstity}.
Observing that $L(e_{\lambda})=\widetilde{L}(e_{\lambda})$, we have by 
applying Proposition \ref{prop:inproof:1} and \ref{prop:inproof:3}  
\begin{align*}
& \int_{\mathbb{R}} f(x) \Em_{x}\left(e^{- L(e_{\lambda})} g(X(e_{\lambda})); e_{\lambda}\leq \tdc\right)\,dx\\
=&\ \iint f(x) g(z) \Em_{x}\left(e^{-L(e_{\lambda})}; e_{\lambda}\leq \tdc\middle|X_{e_{\lambda}}=z\right) \Pm_{x}(X_{e_{\lambda}}\in\,dz)\,dx\\
=&\ \iint f(x) g(z) \Em_{x}\left(e^{-\widetilde{L}(e_{\lambda})}; e_{\lambda}\leq \widetilde{\tau}_{c}^{+}\wedge \widetilde{\tau}_{0}^{-}\middle|X_{e_{\lambda}}=z\right) \Pm_{x}(X_{e_{\lambda}}\in\,dz)\,dx\\
=&\ \iint f(x) g(z) \widehat{\Em}_{z}\left(e^{-L(e_{\lambda})}; e_{\lambda}\leq \tdc\middle|X_{e_{\lambda}}=x\right) \widehat{\Pm}_{z}(X_{e_{\lambda}}\in\,dx)\,dz\\
=&\  \int_{\mathbb{R}} g(z) \widehat{\Em}_{z} \left(e^{-L(e_{\lambda})} f(X(e_{\lambda})); t\leq \tdc \right) \,dz.
\end{align*}
Equation \eqref{eqn:inproof:dual} is proved by applying Theorem \ref{main:resol}.

On the set $\{\sxp>0\}=\{\tau_{0}^{+}<e_{\lambda}\}=\{\widetilde{\tau}_{0}^{+}<e_{\lambda}\}$, we have
\[
\sxp +\widetilde{\tau}_{0}^{+}=e_{\lambda},\quad
X(\sxp-)=\widetilde{X}(\widetilde{\tau}_{0}^{+}),\quad 
L(\sxp)=\widetilde{L}(e_{\lambda})- \widetilde{L}(\widetilde{\tau}_{0}^{+}).
\]
Moreover $\{0<\sxp \leq \tdc\}=\{ (\widetilde{\tau}_{d}^{+}\wedge \widetilde{\tau}_{c}^{-})\circ\theta_{\widetilde{\tau}_{0}^{+}}\geq e_{\lambda}-\widetilde{\tau}_{0}^{+}>0\}$
by definitions,  where $\theta_{\cdot}$ is the shifting operator. The event on the righthand side means that after $\widetilde{\tau}_{0}^{+}$, $\widetilde{X}$ doesn't exit $[c,d]$ before $e_{\lambda}$. 
Therefore, we have for $z\in\mathbb{R}$
\begin{align*}
&\  \Em_{x}\left[e^{-L(\sxp)} g(X(\sxp-)); 0< \sxp\leq \tdc\middle|X_{e_{\lambda}}=z\right]\\
=&\ \Em_{x}\left[e^{\widetilde{L}(\widetilde{\tau}_{0}^{+})-\widetilde{L}(e_{\lambda})} g(\widetilde{X}(\widetilde{\tau}_{0}^{+}));  
(\widetilde{\tau}_{d}^{+}\wedge \widetilde{\tau}_{c}^{-})\circ\theta_{\widetilde{\tau}_{0}^{+}}\geq e_{\lambda}-\widetilde{\tau}_{0}^{+}>0\middle|X_{e_{\lambda}}=z\right]\\
=&\ \widehat{\Em}_{z} \left[e^{L(\tau_{0}^{+})-L(e_{\lambda})} g(X(\tau_{0}^{+})); 
(\tdc)\circ\theta_{\tau_{0}^{+}}\geq e_{\lambda}-\tau_{0}^{+}>0
\middle|X_{e_{\lambda}}=x\right],
\end{align*}
where Proposition \ref{prop:inproof:3} is applied in the last line. Therefore employing Proposition \ref{prop:inproof:1}, we have
\begin{align*}
&\ \int f(x) \Em_{x}\left[e^{-L(\sxp)} g(X(\sxp-)); 0< \sxp\leq \tdc\right] dx\\
=&\ \iint f(x) \Em_{x}\left[e^{-L(\sxp)} g(X(\sxp-)); 0< \sxp\leq \tdc\middle|X_{e_{\lambda}}=z\right]\Pm_{x}(X_{e_{\lambda}}\in dz) dx\\
=&\ \iint  f(x) \widehat{\Em}_{z} \left[e^{L(\tau_{0}^{+})-L(e_{\lambda})} g(X_{\tau_{0}^{+}}); 
(\tdc)\circ\theta_{\tau_{0}^{+}}\geq e_{\lambda}-\tau_{0}^{+}>0
\middle|X_{e_{\lambda}}=x\right] \widehat{\Pm}_{z}(X_{e_{\lambda}}\in dx) dz\\
=&\ \int \widehat{\Em}_{z} \left[g(X(\tau_{0}^{+})) f(X_{e_{\lambda}}) e^{-L(e_{\lambda}-\tau_{0}^{+})\circ \theta_{\tau_{0}^{+}}}; 
(\tdc)\circ\theta_{\tau_{0}^{+}}\geq e_{\lambda}-\tau_{0}^{+}>0 \right]  dz\\
=&\ \int \widehat{\Em}_{z}\left[g(X(\tau_{0}^{+}))\cdot \widehat{\Em}_{X(\tau_{0}^{+})} \left(e^{-L(e_{\lambda})}  f(X(e_{\lambda})) ; e_{\lambda}\leq \tdc\right); \tau_{0}^{+}<e_{\lambda}\right] dz,
\end{align*}
where the Markov property of $X$ and the memoryless property of $e_{\lambda}$ is used.

Finally, taking advantage of equation \eqref{eqn:inproof:dual}
\begin{align*}
&\ \int  f(x) \Em_{x}\left[e^{-L(\sxp)} g(X(\sxp-)); 0<\sxp\leq \tdc\right] dx\\
=&\ \lambda \int \Em_{z}\left(e^{-\lambda \tau_{0}^{-}}; -X(\tau_{0}^{-})\in\,dy\right)\cdot g(y) \cdot \left(\int f(x)  \uab(x,y)  dx\right)\,dz \\
=&\  \int g(y) \mu^{(\lambda)}(dy)  \int f(x)  \uab(x,y)  dx,
\end{align*}
where $\mu^{(\lambda)}$ is the measure in Lemma \ref{lem:mu}. Thus,
\[
 \Em_{x}\left[e^{-L(\sxp)} g(X(\sxp-)); 0<\sxp\leq \tdc\right]=\int g(y) \uab(x,y), \mu^{(\lambda)}(dy),
\]
and this completes the proof of Theorem \ref{main:s+}.
\end{proof}

Similar approach could be adopted to derive formulas in Theorem \ref{main:s-} and \ref{main:s=}. 
In light of Remark \ref{rmk:2}, we focus more on the event of creeping in the following proofs.

\begin{proof}[Proof of Theorem \ref{main:s-}] 
Noting that being exclusive of positive jumps, $X$ is continuous at $\sxm$. $\{\sxm=e_{\lambda}\}=\{X(e_{\lambda})<0\}$ by definition, then we have  for $y<0$,
\begin{align*}
&\ \Em_{x}\left[e^{-L(\sxm)}; X(\sxm)\in dy, 0<\sxm\leq \tdc\right]\\
=&\ \Em_{x}\left[e^{-L(e_{\lambda})}; X(e_{\lambda})\in dy,  e_{\lambda}\leq \tdc\right]\\
=&\ \lambda \int_{0}^{\infty} e^{-\lambda t} \Em_{x} \left[e^{-L(t)}; X_{t}\in dy, t\leq \tdc\right]= \lambda\cdot \uab(x,y) dy.
\end{align*}

Furthermore, on the set $\{0<\sxm<e_{\lambda}\}=\{X(\sxm)=0\}=\{0<\widetilde{\tau}_{0}^{-}<e_{\lambda}\}$, we have 
\[
\sxm+\widetilde{\tau}_{0}^{-}=e_{\lambda}\quad \text{and}\quad L(\sxm)=\widetilde{L}(e_{\lambda})-\widetilde{L}(\widetilde{\tau}_{0}^{-}).\]
Making use of the same argument as in the previous proof, we will have 
\begin{align*}
&\ \int_{\mathbb{R}} f(x) \Em_{x} \left[e^{-L(\sxm)}; X(\sxm)=0, 0<\sxm\leq \tdc\right] dx\\
=&\ \int \widehat{\Em}_{z}\left[ e^{-L(e_{\lambda})- L(\tau_{0}^{-})} f(X(e_{\lambda})); 
0<e_{\lambda}-\tau_{0}^{-}\leq (\tdc)\circ \theta_{\tau_{0}^{-}}\right]\,dz\\
=&\ \int \left(\widehat{\Pm}_{z}(\tau_{0}^{-}<e_{\lambda}) \cdot \widehat{\Em}\left(e^{-L(e_{\lambda})} f(X(e_{\lambda})); e_{\lambda}\leq \tdc\right) \right)dz \\
=&\  \lambda \int_{z>0} \widehat{\Em}_{z}\left(e^{-\lambda \tau_{0}^{-}} \right) dz  \times \int_{0}^{\infty}e^{-\lambda t}  \widehat{\Em}\left(  e^{-L(t)}f(X_{t}); t\leq \tdc\right) dt,
\end{align*}
where  the Markov property of $X$, memoryless property of $e_{\lambda}$, 
${X}(\tau_{0}^{-})=0$ under $\widehat{\Pm}_{z}$ for $z>0$ 
and $\{\tau_{0}^{-}>0\}=\{X_{0}>0\}$ 
are applied in the last two lines.

Considering equation \eqref{eqn:inproof:dual} we further have
\begin{align*}
&\ \int f(x) \Em_{x} \left[e^{-L(\sxm)}; X(\sxm)=0, 0<\sxm\leq \tdc\right] dx\\
=&\ \lambda \int_{z>0} e^{-\phi(\lambda)z}  dz\int f(x) \uab(x,0) dx=\frac{\lambda}{\phi(\lambda)} \int f(x) \uab(x,0) dx,
\end{align*}
which gives formula \eqref{eqn:main:s-:=}, and this finishes the proof.
\end{proof}

\begin{proof}[Proof of Theorem \ref{main:s=}]  Again, according to the dual argument, we will have
\begin{align*}
&\ \int_{\mathbb{R}} f(x) \Em_{x} \left[e^{-L(\sxe)}; 0<\sxe\leq \tdc\right]  dx\\
=&\  \int_{\mathbb{R}} \widehat{\Em}_{z}\left[e^{-\lambda \tau^{\{0\}}} 
\int_{0}^{\infty} \widehat{\Em}\left(e^{-L(e_{\lambda})} f(X(e_{\lambda})); e_{\lambda}\leq \tdc\right)\right] dz\\
=&\  \left(\int_{\mathbb{R}} f(x) \uab(x,0) dx \right)
\left(\int_{\mathbb{R}} \Em_{z}(e^{-\lambda \tau^{\{0\}}}) dz\right).
\end{align*}
Since for $y\geq0$, $\Em_{-y}[e^{-\lambda \tau_{0}^{+}}]=e^{-\phi(\lambda)y}$, we have from Lemma \ref{lem:mu} that 
\[
\lambda \int_{\mathbb{R}} \Em_{z}(e^{-\lambda \tau^{\{0\}}}) dz=\int_{0}^{\infty} e^{-\phi(\lambda) y} \mu^{(\lambda)}(dy)
=\lim_{s\to\phi(\lambda)}\frac{\phi(\lambda)(\psi(s)- \lambda)}{s(s-\phi(\lambda))}=\psi'(\phi(\lambda)).
\]
Putting them together gives
\[
\Em_{x} \left[e^{-L(\sxe)}; 0<\sxe\leq \tdc\right]= \frac{1}{\phi'(\lambda)} \uab(x,0),
\]
which proves the desired result.
\end{proof}

\begin{proof}[Proof of Corollary \ref{cor:4}] Basically,  the joint distributions of 
$(L(\cdot),X(e_{\lambda}))$ are direct consequences of the dual argument, by \eqref{eqn:inproof:dual} and Proposition \ref{prop:levy}.

For the differences of the last times, on the set $\{0<\sxe<\sxp<e_{\lambda}\}$, we have
\[
\{0<\sxe<\sxp<e_{\lambda}\}=\{\widetilde{\tau}_{0}^{+}<e_{\lambda}\}\cap \{\widetilde{\tau}_{0}^{-}<e_{\lambda}-\widetilde{\tau}_{0}^{+}\leq \widetilde{\tau}_{d}^{+}\wedge\widetilde{\tau}_{c}^{-}\}\circ\theta_{\widetilde{\tau}_{0}^{+}},
\]
\[
\sxp+ \widetilde{\tau}_{0}^{+}=e_{\lambda}, 
\quad \sxe+ \widetilde{\tau}^{\{0\}}=e_{\lambda} 
\quad\text{and}\quad 
L(\sxp)-L(\sxe)= L(\widetilde{\tau}_{0}^{-})\circ \theta_{\widetilde{\tau}_{0}^{+}}.
\]
The dual arguments suggest that, for $x,y,z>0$ with $y\neq x$,
\begin{align*}
&\ \int f(u) \Em_{u}\left[e^{L(\sxe)-L(\sxp)} \bo_{\{0<\sxe<\sxp\leq \tdc\}}; 
X_{\sxp-}\in dz, -X_{\sxp}\in dy, -X_{e_{\lambda}}\in dx\right] du\\
& = \widehat{\Em}_{-x}\left[ f(X_{e_{\lambda}}) 
\left(e^{-L(\tau_{0}^{-})} \bo_{\{\tau_{0}^{-}<e_{\lambda}-\tau_{0}^{+}\leq \tdc\}}\right)\circ\theta_{\tau_{0}^{+}}; -X_{\tau_{0}^{+}-}\in dy, X_{\tau_{0}^{+}}\in dz, 
\tau_{0}^{+}<e_{\lambda}\right] dx.
\end{align*}
Taking account of the Markov property of $X$ and the memoryless property of $e_{\lambda}$, it equals to
\begin{align*}
&\ \widehat{\Pm}_{-x}(-X_{\tau_{0}^{+}-}\in\,dy, X_{\tau_{0}^{+}}\in dz, \tau_{0}^{+}<e_{\lambda})\cdot \widehat{\Em}_{z}\left[e^{-L(\tau_{0}^{-})}f(X_{e_{\lambda}}); 
\tau_{0}^{-}<e_{\lambda}\leq \tdc \right]\,dx\\
=&\ \Em_{x}\left(e^{-\lambda \tau_{0}^{-}}; X_{\tau_{0}^{-}-}\in\,dy, -X_{\tau_{0}^{-}}\in dz\right)
\cdot \widehat{\Em}_{z}\left[e^{-L(\tau_{0}^{-})}; \tau_{0}^{-}<e_{\lambda}\wedge \tdc \right]\\
&\ \quad \times \widehat{\Em}\left(f(X_{e_{\lambda}}); e_{\lambda}\leq \tdc\right)\,dx\\
=&\ R^{(\lambda)}(x,dy)\cdot \Pi(dz+y) \cdot 
\widehat{\Em}_{z}\left[e^{-L(\tau_{0}^{-})}; \tau_{0}^{-}<e_{\lambda}\wedge \tau_{d}^{+} \right]\cdot \left(\lambda \int_{\mathbb{R}} u^{(\lambda)}(u,0) f(u)\,du\right)\,dx,
\end{align*}
where 
$R^{(\lambda)}(x,dy)$ is the resolvent measure as mentioned in Proposition \ref{prop:levy} and 
$u^{(\lambda)}(x,y)=\uab(x,y)$ with $p=q=0$ from identity \eqref{eqn:inproof:dual}.
Since $\D \widehat{\Em}_{z}\left[e^{-L(\tau_{0}^{-})}; \tau_{0}^{-}<e_{\lambda}\wedge \tau_{d}^{+}\right]=\frac{\mathcal{W}_{(a,b)}^{(p+\lambda,q+\lambda)}(d,z)}{\mathcal{W}_{(a,b)}^{(p+\lambda,q+\lambda)}(d,0)}$ as one can check. Coronary \ref{cor:4} is thus proved.
\end{proof}

\section{Conclusions}
Last passage times are as important as first passage times in studying Markov processes and can also find their applications in the risk theory. 
In this paper, 
the occupation times of intervals until last passage times for an SNLP are investigated. 
By applying the dual argument, we obtain the explicit formulas for their Laplace transforms.
The employed method also helps us to provide a characterisation of other features at the moment of last passage times.

\section*{Acknowledgement}
We are grateful for Droctor Weihong Ni at Liverpool University for her fruitful comments and suggestions, which helped improve the presentation. 
Both authors acknowledge financial supports from the National Natural Science Foundation of China (Grant No. 11501304 and No. 11601243).


\end{document}